\documentclass[]{siamart171218}      

\usepackage{amsmath}
\usepackage{amssymb}
\usepackage{tikz}
\usepackage{listings}

\newcommand{\R}{{I\!\!R}}

\newcommand{\X}{{\bf x}}
\newcommand{\Z}{{\bf z}}
\newcommand{\Y}{{\bf y}}

\newcommand{\fma}{{\small {\tt fma}}}

\title{Optimization of Generalized Jacobian Chain Products without Memory Constraints}
\author{Uwe Naumann\thanks{Informatik 12: Software and Tools for Computational Engineering, RWTH Aachen University, Germany. \email{naumann@stce.rwth-aachen.de}}}

\begin{document}

\maketitle


\begin{abstract}
The efficient computation of Jacobians represents a fundamental challenge in 
computational science and engineering. Large-scale modular numerical 
simulation programs can be regarded as sequences of evaluations of in our case
differentiable subprograms with corresponding local Jacobians. The latter are
typically not available. Tangent and adjoint versions of the 
individual subprograms are assumed to be given as results of algorithmic 
differentiation instead. The classical (Jacobian) matrix chain product
formulation is extended with the optional evaluation of
matrix-free Jacobian-matrix and matrix-Jacobian products as tangents and 
adjoints.
We present a dynamic programming algorithm for the solution
	of the {\sc Generalized Dense Jacobian Chain Product} problem without
considering constraints on the available persistent system memory. In other 
words, the naive evaluation of an adjoint of the entire simulation program 
is assumed to be a feasible option. We obtain optimal solutions which 
	improve the best 
state of the art methods by factors of up to sixty on a set of randomly 
generated problem instances of growing size. Our results illustrate the 
	potential of algorithmic differentiation methods beyond tangent and adjoint modes.
\end{abstract}

\section{Introduction}

This paper extends our prior work on computational cost-efficient accumulation 
of Jacobian matrices. The corresponding combinatorial {\sc Optimal Jacobian 
Accumulation} (OJA) problem was shown to be
NP-complete in \cite{Naumann2008OJa}. Elimination techniques yield different
structural variants of OJA discussed in 
\cite{Naumann2004Oao}. Certain special cases turn out to be computationally 
tractable as described in \cite{Griewank2003AJa} and \cite{Naumann2008Ove}.

Relevant closely related work by others includes the introduction of 
{\sc Vertex Elimination} (VE) \cite{Griewank1991OtC}, an integer programming 
approach to VE \cite{Chen2012AIP}, computational experiments with VE 
\cite{Forth2004JCG}, and the formulation of OJA as LU factorization 
\cite{Pryce2008FAD}. 

Let the multivariate vector function
$
\Y=F(\X) : \R^n \rightarrow \R^m
$
(in the following referred to as the {\em primal} function)
be differentiable over the domain of interest and let
$
F=F_q \circ F_{q-1} \circ \ldots \circ F_2 \circ F_1 
$
be such that $\Z_i=F_i(\Z_{i-1}) : \R^{n_i} \rightarrow \R^{m_i}$ for $i=1,\ldots,q$
and $\Z_0=\X,$ $\Y=\Z_q.$
According to the chain rule of differential calculus the Jacobian $F'=F'(\X)$ of $F$ is
equal to
\begin{equation} \label{eqn:jcp}
F' \equiv \frac{d F}{d \X}=F'_q \cdot F'_{q-1} \cdot \ldots \cdot F'_1 \in \R^{m \times n} \; .
\end{equation}
We discuss the minimization of the computational cost in term of {\em fused multiply-add} (\fma) operations of the evaluation of Equation~(\ref{eqn:jcp}).

Algorithmic differentiation \cite{Griewank2008EDP,Naumann2012TAo} 
offers two fundamental modes for {\em preaccumulation} of the local Jacobians 
$F'_i = F'_i(\Z_{i-1}) \in \R^{m_i \times n_i}$ prior to the evaluation of the matrix chain product in Equation~(\ref{eqn:jcp}). Directional derivatives are computed in {\em scalar tangent mode} as
\begin{equation} \label{eqn:st}
\dot{\Z}_i=F'_i \cdot \dot{\Z}_{i-1} \in \R^{m_i} \; .
\end{equation}
Accumulation of the entire Jacobian requires evaluation of $n_i$ tangents 
in the Cartesian basis directions in $\R^{n_i}$
if
$F'_i$ is dense. Potential sparsity can and should be detected \cite{Griewank2002DJS} and exploited \cite{Gebremedhin2005WCI,Hossain2002SIi}. 
We denote the computational cost of evaluating a subchain 
$F'_j \cdot \ldots \cdot F'_i,$ $j>i,$ of Equation~(\ref{eqn:jcp}) 
as $\fma_{j,i}$. The computational cost of evaluating $F'_i$ in tangent 
mode is denoted as $\fma_{i,i}=\dot{\fma}_i$.

{\em Scalar Adjoint mode} yields
\begin{equation} \label{eqn:sa}
\bar{\Z}_{i-1}= \bar{\Z}_i \cdot F'_i \in \R^{1 \times n_i} 
\end{equation}
and hence dense Jacobians by $m_i$ reevaluations of Equation~(\ref{eqn:sa}) with
$\bar{\Z}_i$ ranging over the Cartesian basis directions in $\R^{m_i}.$ 
The scalar adjoint of $\Z_i$ can be interpreted as the derivative 
of some scalar objective with respect to $\Z_i$ yielding
$\bar{\Z}_i \in \R^{1 \times m_i}$ as a row vector.
The computational cost of evaluating $F'_i$ in adjoint
mode is denoted as $\fma_{i,i}=\bar{\fma}_i$. 
Further formalization of this cost 
estimate will follow in Section~\ref{sec:gjcp}. Combinatorially more 
challenging Jacobian accumulation methods based on elimination techniques 
applied to computational graphs \cite{Naumann2004Oao} will not be considered 
here.
While they may yield a further reduction of $\fma_{i,i}$ the resulting 
irregularity of memory accesses makes actual gains in 
computational performance hard to achieve.

The {\sc Jacobian Chain Product Bracketing} problem asks for a 
bracketing of the right-hand side of Equation~(\ref{eqn:jcp})
which minimizes the number of \fma\ operations.
{\sc Jacobian Chain Product Bracketing} can be solved by dynamic programming 
\cite{Bellman1957DP,Godbole1973} 
even if the individual factors
are sparse. Sparsity patterns of all subproducts need to be evaluated
symbolically in this case \cite{Griewank2003AJa}.
The following recurrence yields an optimal bracketing at a computational cost
of $O(q^3):$
\begin{equation} \label{eqn:dp1}
\fma_{j,i}=
\begin{cases}
\min(\dot{\fma}_i,\bar{\fma}_i) & j=i \\ 
\min_{i \leq k < j} \left (\fma_{j,k+1}+\fma_{k,i} + \fma_{j,k,i} \right ) & j>i \; .
\end{cases}
\end{equation}
Facilitated by the
{\em overlapping subproblems} and {\em optimal substructure} 
properties of {\sc Jacobian Chain Product Bracketing} the optimization of enclosing chains
look up tabulated solutions to all subproblems at constant time complexity.
For example, a dense Jacobian chain product of length $q=4$ with 
$F'_4 \in \R^4,$ $F'_3 \in \R^{1 \times 4},$ $F'_2 \in \R^{4 \times 5},$ $F'_1 \in \R^{5 \times 3}$ and
$\fma_{4,4}=21,$ 
$\fma_{3,3}=5,$ 
$\fma_{2,2}=192,$ 
$\fma_{1,1}=84$ yields the optimal bracketing
$
F'=F'_4 \cdot ((F'_3 \cdot F'_2) \cdot F'_1)
$
with a cumulative cost of $ 349\;\fma.$ 

The more general {\sc Jacobian Chain Product} problem asks for some \fma-optimal
way 
to compute $F'$ without the restriction of the search space to valid 
bracketings of Equation~(\ref{eqn:jcp}). For example, the matrix product
$$
\begin{pmatrix}
6 & 0 \\
0 & 7 \\
\end{pmatrix}
\begin{pmatrix}
7 & 0 \\
0 & 6 \\
\end{pmatrix}=
\begin{pmatrix}
42 & 0 \\
0 & 42 \\
\end{pmatrix}
$$
\cite{Adams1979THG} can be evaluated at the expense of a single \fma\ as opposed 
to two by exploiting commutativity of scalar multiplication.
{\sc Jacobian Chain Product} is known to be NP-complete; see \cite{Naumann2020CSC} as well as the upcoming proof 
of Theorem~\ref{the:gjcp_np}.

\section{\sc Generalized Jacobian Chain Product} \label{sec:gjcp}

Any $F_i=F_i(\Z_{i-1})$ induces a labeled directed acyclic graph (DAG) $G_i=G_i(\Z_{i-1})=(V_i,E_i)$ 
for $i=1,\ldots,q.$ 
Vertices in $V_i=\{v^i_j : j=1,\ldots,|V_i|\}$ represent the elemental arithmetic operations $\varphi^i_j \in \{+,\sin,\ldots\}$ executed by the implementation of $F_i$ for given $\Z_{i-1}$. Edges in $(j,k) \in E_i \subseteq V_i \times V_i$ 
mark data dependencies between arguments and results of elemental operations. They are labeled with local partial derivatives $$
\frac{\partial \varphi^i_k}{\partial v^i_j} \; , \quad k:~(j,k) \in E_i
$$ of the 
elemental functions with respect to their arguments. An example is shown in Figure~\ref{fig:ex1}. Note that a single evaluation of the adjoint in 
Figure~\ref{fig:ex1}~(d) delivers both gradient entries for $\bar{z}_1^{i+1}=1$ 
while two evaluations of the tangent with $\dot{\Z}^i=(1~0)^T$ and $\dot{\Z}^i=(0~1)^T$ are required to complete the same task.
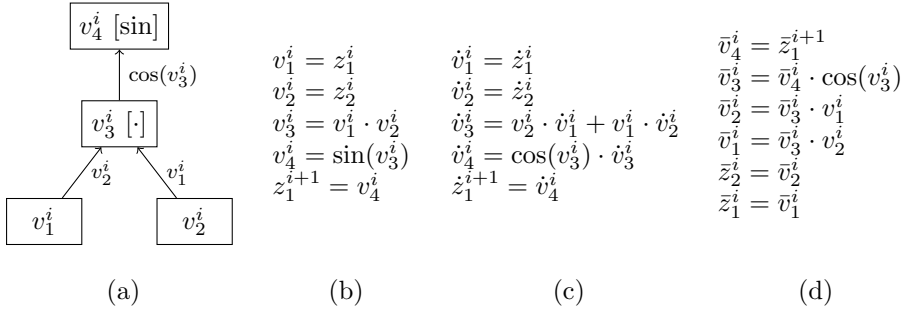
\begin{figure}[hbt]
\begin{tabular}{cccc}
\begin{minipage}[c]{.24\linewidth}
\begin{tikzpicture}[scale=1, transform shape, rectangle]
  \begin{pgfscope}
\tikzstyle{every node}=[draw,rectangle,minimum width=1cm]
	  \node (0) at (0,0) {$v^i_1$};
	  \node (1) at (2,0) {$v^i_2$};
	  \node (2) at (1,1.3) {$v^i_3~[\cdot]$};
	  \node (3) at (1,2.6) {$v^i_4~[\sin]$};
  \end{pgfscope}
  \begin{scope}[->]
          \draw (0) -- (2) node[midway,right] {\footnotesize $v^i_2$}; 
          \draw (1) -- (2) node[midway,right] {\footnotesize $v^i_1$}; 
          \draw (2) -- (3) node[midway,right] {\footnotesize $\cos(v^i_3)$}; 
  \end{scope}
\end{tikzpicture}

\end{minipage} &
\begin{minipage}[c]{.15\linewidth}
$v_1^i=z_1^i $ \\
$v_2^i=z_2^i $ \\
$v_3^i=v_1^i \cdot v_2^i$ \\
$v_4^i=\sin(v_3^i)$ \\
$z_1^{i+1}=v_4^i$ 
\end{minipage} &
\begin{minipage}[c]{.24\linewidth}
$\dot{v}_1^i=\dot{z}_{1}^i $ \\
$\dot{v}_2^i=\dot{z}_2^i $ \\
$\dot{v}_3^i=v_2^i \cdot \dot{v}_1^i +v_1^i \cdot \dot{v}_2^i$ \\
$\dot{v}_4^i=\cos(v_3^i) \cdot \dot{v}_3^i$ \\
$\dot{z}_1^{i+1}=\dot{v}_4^i$ 
\end{minipage} &
\begin{minipage}[c]{.2\linewidth}
$\bar{v}_4^i=\bar{z}_1^{i+1}$ \\
$\bar{v}_3^i=\bar{v}_4^i \cdot \cos(v_3^i)$ \\
$\bar{v}_2^i=\bar{v}_3^i \cdot v_1^i$ \\
$\bar{v}_1^i=\bar{v}_3^i \cdot v_2^i$ \\
$\bar{z}_2^i=\bar{v}_2^i$ \\
$\bar{z}_1^i=\bar{v}_1^i$ 
\end{minipage} \\
\\
(a) & (b) & (c) & (d)
\end{tabular}
\caption{Simple Example: Labeled DAG (a); primal (b); scalar tangent (c); scalar adjoint (d)} \label{fig:ex1}
\end{figure}

Preaccumulation of local Jacobians $F'_i \in \R^{m_i \times n_i}$ requires
either $n_i$ evaluations of the scalar tangent or 
$m_i$ evaluations of the scalar adjoint. In order to avoid unnecessary
reevaluation of the function values and of the local partial derivatives we 
switch to vectorized versions of tangent and adjoint modes.

For given 
$\Z_{i-1} \in \R^{n_i}$ and
$\dot{Z}_{i-1} \in \R^{n_i \times \dot{n}_i}$ the
Jacobian-free evaluation of $$\dot{Z}_i = F_i'(\Z_{i-1}) \cdot \dot{Z}_{i-1} \in \R^{m_i \times \dot{n}_i}$$ 
in {\em vector tangent mode} is denoted as 
\begin{equation} \label{eqn:vt}
\dot{Z}_i := \dot{F}_i(\Z_{i-1}) \cdot \dot{Z}_{i-1} \; .
\end{equation} 
Preaccumulation of a dense $F'_i$ requires $\dot{Z}_{i-1}$ to be equal to the 
identity $I_{n_i} \in \R^{n_i \times n_i}.$
Equation~(\ref{eqn:vt}) amounts to the simultaneous propagation of $\dot{n}_i$ 
tangents through $G_i.$ Explicit construction
(and storage) of $G_i$ is not required as the computation of tangents  
augments the primal arithmetic locally. For example, the codes in 
Figure~\ref{fig:ex1}~(b) and (c) can be interleaved as 
$v^i_j=\ldots;~\dot{v}^i_j=\ldots$ for 
$j=1,\ldots,4.$
Tangent propagation induces a computational cost of 
\mbox{$\dot{n}_i \cdot |E_i|$} in addition to the invariant cost of the primal function
evaluation ($|V_i|$) augmented with the computation of all local partial 
derivatives ($|E_i|$). 
The invariant memory requirement of the primal function evaluation is increased by the memory requirement of the tangents
the minimization of which turns out to be NP complete as a variant of the
{\sc Directed Bandwidth} problem \cite{Naumann2018LMA}.
In the following the invariant part of the computational 
cost will not be included in cost estimates. The memory requirements 
of all instances of the discrete search spaces of the combinatorial optimization
problems considered in this paper are assumed to be feasible.

Equation~(\ref{eqn:vt}) can be interpreted as the ``product'' of the DAG $G_i$ 
with the matrix $\dot{Z}_{i-1}.$ 
If $\dot{Z}_{i-1}$ is dense, then its DAG becomes the directed acyclic
version of the complete bipartite graph $K_{n_{i-1},m_{i-1}}.$ The computation
of $\dot{Z}_i$ amounts to the application of the chain rule to the composite
DAG \cite{Baur1983TCo}. It can be interpreted as forward vertex elimination \cite{Griewank1991OtC} yielding a computational cost of \mbox{$\dot{n}_{i-1} \cdot |E_i|$}.
 
For given 
$\Z_{i-1} \in \R^{n_i}$ and
$\bar{Z}_i \in \R^{\bar{m}_i \times m_i}$ the
Jacobian-free evaluation of $$\bar{Z}_{i-1} = \bar{Z}_i \cdot F_i'(\Z_{i-1}) \in \R^{\bar{m}_i \times n_i}$$ 
in {\em vector adjoint mode} 
is denoted as 
\begin{equation} \label{eqn:va}
\bar{Z}_{i-1} := \bar{Z}_i \cdot \bar{F}_i(\Z_{i-1}) \; .
\end{equation} 
Preaccumulation of a dense $F'_i$ requires $\bar{Z}_i$ to be equal to the 
identity
$I_{m_i} \in \R^{m_i \times m_i}.$
Equation~(\ref{eqn:va}) represents the simultaneous reverse propagation of 
$\bar{m}_i$ adjoints through $G_i.$ Without constraints on
the total memory requirement  
the cost-optimal propagation of adjoints amounts to storage of $G_i$ 
thus avoiding unnecessary reevaluation of (parts of) the primal function in the context of {\em checkpointing} methods \cite{Griewank1992ALG}. 
For example, the reversal of the data flow requires the adjoint code in 
Figure~\ref{fig:ex1}~(d) to be preceded by (the relevant parts of) the 
primal code in Figure~\ref{fig:ex1}~(b) (computation of $v_3$).
Vector adjoint propagation induces an (additional) computational cost of 
$\bar{m}_i \cdot |E_i|$.
The minimization of the additional memory requirement 
amounts to a variant of the NP complete {\sc Directed Bandwidth} problem \cite{Naumann2018LMA}.

Equation~(\ref{eqn:va}) can be interpreted as the ``product'' of 
the matrix $\bar{Z}_i$ with the DAG $G_i.$ 
If $\bar{Z}_i$ is dense, then its DAG becomes the directed acyclic
version of the complete bipartite graph $K_{n_i,m_i}.$ The computation
of $\bar{Z}_{i-1}$ amounts to the application of the chain rule to the 
composite DAG \cite{Baur1983TCo}. It can be interpreted as backward vertex elimination 
\cite{Griewank1991OtC} yielding a computational cost of 
\mbox{$\bar{m}_i \cdot |E_i|$}.

Vector tangent and vector adjoint modes belong to the fundamental set of
functionalities offered by the majority of mature algorithmic differentiation software solutions.
Hence, we assume them to be available for all $F_i$ and we refer to them
simply as tangents and adjoints.

Analogous to {\sc Jacobian Chain Product} 
the {\sc Generalized Jacobian Chain Product} problem
asks for an algorithm for computing $F'$ with a minimum number of 
\fma\ operations for
given tangents and adjoints for all $F_i$ in Equation~(\ref{eqn:jcp}).
As a generalization of an NP-complete problem {\sc Generalized Jacobian Chain Product} must be computationally 
intractable too. The corresponding proof turns out to be very similar to the
arguments presented in \cite{Naumann2008OJa} and \cite{Naumann2020CSC}. 
It uses reduction from {\sc Ensemble Computation} which was shown to be 
NP-complete in \cite{Garey1979CaI}:

Given a collection
$C = \{C_\nu \subseteq A : \nu=1,\ldots,|C|\}$ of subsets
$C_\nu = \{c_i^\nu:i=1,\ldots,|C_\nu|\}$
of a finite set $A$ and a positive integer
$K$ is there a sequence
$u_i=s_i \cup t_i$ for $i=1,\ldots,k$ of $k \leq K$ union
operations, where each $s_i$ and $t_i$ is either $\{a\}$ for some $a \in A$
or $u_j$
for some
$j < i,$ such that $s_i$ and $t_i$ are disjoint for $i=1,\ldots,k$ and
such that for every subset $C_\nu \in C,$ $\nu=1,\ldots,|C|,$
there is some $u_i,$ $1 \leq i \leq k,$ that is identical to $C_\nu.$
Instances of {\sc Ensemble Computation} are given as triplets $(A,C,K).$

For example, let $A=\{a_1,a_2,a_3,a_4\},$ $C=\left \{\{a_1,a_2\},\{a_2,a_3,a_4\},\{a_1,a_3,a_4\}\right \}$ and $K=4.$ The answer to the decision version of this instance of 
{\sc Ensemble Computation} is positive with a corresponding solution given by
$C_1=u_1=\{a_1\} \cup \{a_2\};$ $u_2=\{a_3\} \cup \{a_4\};$ $C_2=u_3=\{a_2\} \cup u_2;$ $C_3=u_4=\{a_1\} \cup u_2.$
$K=3$ yields a negative answer identifying $K=4$ as the solution of
the corresponding minimization version of {\sc Ensemble Computation}.

A decision version of {\sc Generalized Jacobian Chain Product} can be 
formulated as follows:

Let tangents $\dot{F}_i \cdot \dot{Z}_i$ and adjoints 
$\bar{Z}_{i+1} \cdot \bar{F}_i$ be given for all elemental functions
$F_i,$ $i=1,\ldots,q,$ in Equation~(\ref{eqn:jcp}) as well as a positive integer $K.$ 
Is there a sequence of \fma\ operations of length $k\leq K$ which yields 
all nonzero entries of $F'$?

An example can be found in Figure~\ref{fig:red} with further explanation to follow.

\begin{theorem} \label{the:gjcp_np}
{\sc Generalized Jacobian Chain Product} is NP-complete.
\end{theorem}
\begin{proof}
Consider an arbitrary instance $(A,C,K)$ of {\sc Ensemble Computation} and a bijection
$A \leftrightarrow \tilde{A},$ where $\tilde{A}$ consists of $|A|$ mutually 
distinct primes.
A corresponding bijection 
$C \leftrightarrow \tilde{C}$ is implied.
Create an extension $(\tilde{A} \cup \tilde{B},\tilde{C},K+|\tilde{B}|)$ 
by adding unique entries from a sufficiently large set $\tilde{B}$ 
of primes not in $\tilde{A}$ to the $\tilde{C}_j$ such that they all have 
the same cardinality $q$. Note that a solution for this extended 
instance of {\sc Ensemble Computation} implies a solution of the original instance of {\sc Ensemble Computation} as each 
entry of $\tilde{B}$ appears exactly once.

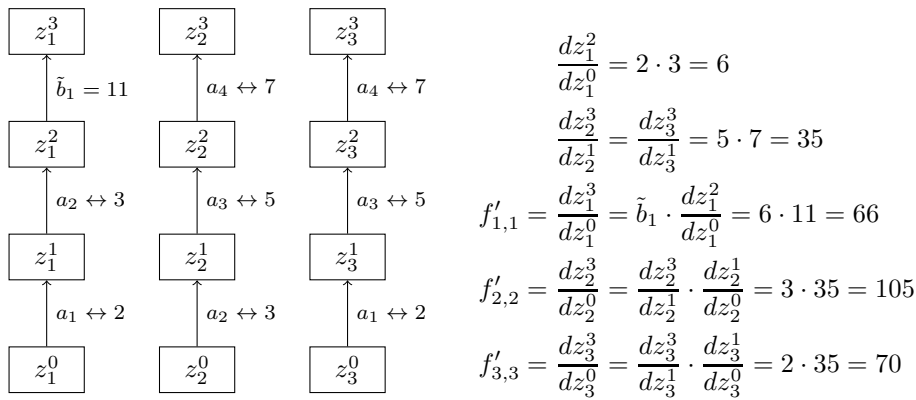
\begin{figure}
\begin{minipage}[c]{.52\linewidth}
\centering
\begin{tikzpicture}[scale=1, transform shape, rectangle]
  \begin{pgfscope}
\tikzstyle{every node}=[draw,rectangle,minimum width=1cm]
	  \node (0) at (0,0) {$z^0_1$};
	  \node (1) at (0,1.5) {$z^1_1$};
	  \node (2) at (0,3) {$z^2_1$};
	  \node (3) at (0,4.5) {$z^3_1$};
	  \node (4) at (2,0) {$z^0_2$};
	  \node (5) at (2,1.5) {$z^1_2$};
	  \node (6) at (2,3) {$z^2_2$};
	  \node (7) at (2,4.5) {$z^3_2$};
	  \node (8) at (4,0) {$z^0_3$};
	  \node (9) at (4,1.5) {$z^1_3$};
	  \node (10) at (4,3) {$z^2_3$};
	  \node (11) at (4,4.5) {$z^3_3$};
  \end{pgfscope}
  \begin{scope}[->]
          \draw (0) -- (1) node[midway,right] {\footnotesize $a_1 \leftrightarrow 2$}; 
          \draw (1) -- (2) node[midway,right] {\footnotesize $a_2 \leftrightarrow 3$}; 
          \draw (2) -- (3) node[midway,right] {\footnotesize $\tilde{b}_1=11$}; 
          \draw (4) -- (5) node[midway,right] {\footnotesize $a_2 \leftrightarrow 3$}; 
          \draw (5) -- (6) node[midway,right] {\footnotesize $a_3 \leftrightarrow 5$}; 
          \draw (6) -- (7) node[midway,right] {\footnotesize $a_4 \leftrightarrow 7$}; 
          \draw (8) -- (9) node[midway,right] {\footnotesize $a_1 \leftrightarrow 2$}; 
          \draw (9) -- (10) node[midway,right] {\footnotesize $a_3 \leftrightarrow 5$}; 
          \draw (10) -- (11) node[midway,right] {\footnotesize $a_4 \leftrightarrow 7$}; 
  \end{scope}
\end{tikzpicture}
\end{minipage} 
\begin{minipage}[c]{.45\linewidth}
\begin{align*}
\frac{d z^2_1}{d z^0_1}&=2 \cdot 3 =6 \\
\frac{d z^3_2}{d z^1_2}&=\frac{d z^3_3}{d z^1_3}=5 \cdot 7 =35 \\
f'_{1,1}=\frac{d z^3_1}{d z^0_1}&=\tilde{b}_1 \cdot \frac{d z^2_1}{d z^0_1}=6 \cdot 11 =66 \\
f'_{2,2}=\frac{d z^3_2}{d z^0_2}&=\frac{d z^3_2}{d z^1_2} \cdot \frac{d z^1_2}{d z^0_2}=3 \cdot 35=105 \\
f'_{3,3}=\frac{d z^3_3}{d z^0_3}&=\frac{d z^3_3}{d z^1_3} \cdot \frac{d z^1_3}{d z^0_3}=2 \cdot 35=70
\end{align*}
\end{minipage} 
\caption{Reduction from {\sc Ensemble Computation} to {\sc (Generalized) Jacobian Chain Product}} 

\label{fig:red}
\end{figure}
Fix the order of the elements of the 
$\tilde{C}_j$ arbitrarily yielding
$\tilde{C}_j=(\tilde{c}^j_i)_{i=1}^q$ for $j=1,\ldots,|\tilde{C}|.$
Let
$$
F_i : \R^{|\tilde{C}|} \rightarrow \R^{|\tilde{C}|} : \quad \Z_i=F_i(\Z_{i-1}) : \quad 
z^i_j=\tilde{c}^j_i \cdot z^{i-1}_j \; .
$$
Equation~(\ref{eqn:jcp}) becomes a diagonal matrix chain product
$F'=F'_q \cdot \ldots \cdot F'_1=D_q \cdot \ldots \cdot D_1$ with
$d^i_{j,j}=\tilde{c}^j_i$ for $j=1,\ldots,|\tilde{C}|$ and 
$i=1,\ldots,q.$ By construction $\fma_{i,i}=\dot{\fma}_i=\bar{\fma}_i=0$
through exploitation of bibpartiteness of the $G_i.$
According to the fundamental theorem of arithmetic \cite{Gauss1801DA}
the elements of $\tilde{C}$ 
correspond to unique (up to commutativity of scalar multiplication) 
factorizations of the $|\tilde{C}|$ nonzero diagonal entries of $F'.$
This uniqueness property extends to arbitrary subsets
of the $\tilde{C}_j$ considered during the exploration of the search space 
of the {\sc Generalized Jacobian Chain Product} problem.

Note that the given family of problem instances are also  
instances of {\sc Jacobian Chain Product} as $\fma_{i,i}=0$ for $i=1,\ldots,q.$ 
A solution implies a solution of the associated extended 
instance of {\sc Ensemble Computation} and, hence, of the original instance 
of {\sc Ensemble Computation}. 

A proposed solution for {\sc Generalized Jacobian Chain Product} is easily 
validated by counting the at
most $|\tilde{C}|\cdot q$ scalar multiplications performed.
\end{proof}

A graphical illustration of the reduction is given in Figure~\ref{fig:red} 
for a problem instance that corresponds to the example presented for 
{\sc Ensemble Computation}. 
\begin{align*}
&A=\{a_1,a_2,a_3,a_4\} \Rightarrow \tilde{A}=\{2,3,5,7\} \\
&\tilde{B}=\{11\} \\
&C=\{\{a_1,a_2\},\{a_2,a_3,a_4\},\{a_1,a_3,a_4\}\} \Rightarrow \tilde{C}=\left \{\{2, 3, 11\},\{3, 5, 7\}, \{2, 5, 7\}\right \} \\
&K+|\tilde{B}|=K+1=5\; .
\end{align*}
The three nonzero diagonal entries of 
$F'=(f'_{j,i}) \in \R^{3 \times 3}$ are computed at the expense of 5 \fma\ (no additions involved)
yielding a positive answer to the given decision version of 
{\sc Generalized Jacobian Chain Product}. 

\section{\sc Generalized Dense Jacobian Chain Product Bracketing}

The formulation of {\sc Jacobian Chain Product Bracketing} assumes availability 
of all factors of 
the Jacobian chain product $F' =F'_q \cdot F'_{q-1} \cdot \ldots \cdot F'_1.$
Locally, the choice is between multiplying $F'_i$ with a factor on its left or on its right within the chain.
{\sc Generalized Dense Jacobian Chain Product Bracketing} assumes availability of 
implementations of 
tangents $\dot{F}_i \cdot \dot{Z}_i$ and adjoints 
$\bar{Z}_{i+1} \cdot \bar{F}_i.$ 
The number of local choices increases. Tangents or adjoints of $F_i$ can be
evaluated or either of them can be used to preaccumulate $F'_i.$ 
All $F'_i$ are assumed to be dense.

Formally, the {\sc Generalized Dense Jacobian Chain Product Bracketing} 
reads as follows:

Let tangents $\dot{F}_i \cdot \dot{Z}_i$ and adjoints 
$\bar{Z}_{i+1} \cdot \bar{F}_i$ be given for all elemental functions
$F_i,$ $i=1,\ldots,q,$ in Equation~(\ref{eqn:jcp}) 
whose respective Jacobians are assumed to be dense.
For a given positive integer $K$ is there a sequence of evaluations of
the tangents and/or adjoints such that the number of \fma\ operations 
required for the accumulation of the Jacobian $F'$ undercuts K?

\paragraph{Example} 
A generalized dense Jacobian chain product of length two yields the following 
eight different bracketings:
\begin{align*}
	F'&=\dot{F}_2 \cdot F'_1 =\dot{F}_2 \cdot (\dot{F}_1  \cdot I_{n_0}) \\
	&=\dot{F}_2 \cdot F'_1 =\dot{F}_2 \cdot (I_{n_1} \cdot \bar{F}_1) \\
	&=F'_2 \cdot \bar{F}_1=(I_{n_2} \cdot \bar{F}_2) \cdot \bar{F}_1 \\
	&=F'_2 \cdot \bar{F}_1=(\dot{F}_2 \cdot I_{n_1}) \cdot \bar{F}_1 \\
	&=F'_2 \cdot F'_1=(\dot{F}_2 \cdot I_{n_1}) \cdot (I_{n_1} \cdot \bar{F}_1) \\
	&=F'_2 \cdot F'_1=(I_{n_2} \cdot \bar{F}_2) \cdot (I_{n_1} \cdot \bar{F}_1) \\
	&=F'_2 \cdot F'_1=(I_{n_2} \cdot \bar{F}_2) \cdot (\dot{F}_1 \cdot I_{n_0}) \\
	&=F'_2 \cdot F'_1=(\dot{F}_2 \cdot I_{n_1}) \cdot (\dot{F}_1 \cdot I_{n_0}) \; .\\
\end{align*}

\begin{theorem} \label{the:GDJCPB}
A solution to {\sc Generalized Dense Jacobian Chain Product Bracketing} can be
computed by dynamic programming as follows:
\begin{equation} \label{eqn:dp2}
\fma_{j,i}=
\begin{cases}
	\centering |E_j| \cdot \min\{n_j,m_j\} & j=i \\ \\
	\min_{i \leq k < j} \left \{ \min \left \{
\begin{split}
&\fma_{j,k+1}+\fma_{k,i} + m_j \cdot m_k \cdot n_i, \\
	&\fma_{j,k+1} + m_j \cdot \sum_{\nu=i}^k |E_\nu|, \\ 
	&\fma_{k,i} + n_i \cdot \sum_{\nu=k+1}^j |E_\nu| 
\end{split}
	\right \} \right \} & j>i \; .
\end{cases}
\end{equation}
\end{theorem}
\begin{proof}
	We enumerate the four different options in 
	Equation~(\ref{eqn:dp2}) as
	\begin{itemize}
		\item[(a)] $|E_j| \cdot \min\{n_j,m_j\},$
		\item[(b)] $\min_{i \leq  k < j} \fma_{j,k+1}+\fma_{k,i} + m_j \cdot m_k \cdot n_i,$
		\item[(c)] $\min_{i\leq   k < j} \fma_{j,k+1} + m_j \cdot \sum_{\nu=i}^k |E_\nu|$ and
		\item[(d)] $\min_{i \leq k <j} \fma_{k,i} + n_i \cdot \sum_{\nu=k+1}^j |E_\nu|.$
			 \end{itemize}
The proof proceeds by induction over $l=j-i.$
\paragraph{$0 \leq l \leq 1$}
All local Jacobians $F'_j=F'_{j,j}$ 
need to be computed in either tangent or adjoint modes at
computational costs of 
$n_j \cdot |E_j|$ or 
$m_j \cdot |E_j|.$ 
The respective minima
are tabulated. Special structure of the underlying DAGs $G_i=(V_i,E_i)$ such
as bipartiteness is not exploited. It could result in lower values for 
$\fma_{i,i},$ e.g, zero in case of bipartiteness.

The search space for 
the product of two dense Jacobians $F'_{i+1} \cdot F'_i$ for given tangents 
and adjoints
of $F_{i+1}$ and $F_i$ 
consists of the following configurations:
\begin{enumerate}
\item $\dot{F}_{i+1} \cdot (\dot{F}_i \cdot I_{n_i}):$ Homogeneous tangent mode
	yields a computational cost of
$$
	\fma_{i+1,i}=n_{i} \cdot |E_{i}|+ n_i \cdot |E_{i+1}| \; .
$$
Equivalently, this scenario can be interpreted as preaccumulation of $F'_i$ in tangent mode followed by evaluation of 
$\dot{F}_{i+1} \cdot F'_i.$ 
This case is covered by Equation~\ref{eqn:dp2}~(a) and (d)
with $n_i \leq m_i.$
\item $\dot{F}_{i+1} \cdot (I_{m_i} \cdot \bar{F}_i):$ Preaccumulation of 
$F'_i$ in adjoint mode followed by evaluation of 
$\dot{F}_{i+1} \cdot F'_i$ yields a computational cost of 
$$
	\fma_{i+1,i}=m_{i} \cdot |E_{i}|+ n_i \cdot |E_{i+1}| \; . 
$$
This case is covered by Equation~\ref{eqn:dp2}~(a) and (d) with
$n_i \geq m_i.$
\item $(I_{m_{i+1}} \cdot \bar{F}_{i+1}) \cdot \bar{F}_i:$
Homogeneous adjoint mode yields a computational cost of
$$
	\fma_{i+1,i}=m_{i+1} \cdot |E_{i+1}|+ m_{i+1} \cdot |E_i| \; .
$$
Equivalently, the preaccumulation of $F'_{i+1}$ in adjoint mode is followed by evaluation of 
$F'_{i+1} \cdot \bar{F}_i.$ 
This case is covered by Equation~\ref{eqn:dp2}~(a) and (c) with
$n_{i+1} \geq m_{i+1}.$
\item $(\dot{F}_{i+1} \cdot I_{n_{i+1}}) \cdot \bar{F}_i:$ 
Preaccumulation of 
		$F'_{i+1}$ in tangent mode followed by evaluation of 
$F'_{i+1} \cdot \bar{F}_i$ yields a computational cost of
$$
	\fma_{i+1,i}=n_{i+1} \cdot |E_{i+1}|+ m_{i+1} \cdot |E_i| \; .
$$
This case is covered by Equation~\ref{eqn:dp2}~(a) and (c) with
$n_{i+1} \leq m_{i+1}.$
\item $(\dot{F}_{i+1} \cdot I_{n_{i+1}}) \cdot (I_{m_i} \cdot \bar{F}_i):$
Preaccumulation of $F'_i$ in adjoint mode followed by 
		preaccumulation of $F'_{i+1}$ in tangent mode and
		evaluation of the dense matrix product
$F'_{i+1} \cdot F'_i$ yields a variant of homogeneous preaccumulation with
a computational cost of
$$
	\fma_{i+1,i}=m_{i} \cdot |E_{i}|+n_{i+1} \cdot |E_{i+1}|+m_{i+1} \cdot n_{i+1} \cdot n_i \; .
$$
This case is covered by Equation~\ref{eqn:dp2}~(a) and (b) with
$n_i \geq m_i$ and $n_{i+1} \leq m_{i+1}.$
\end{enumerate}
The remaining three homogeneous preaccumulation options cannot improve the
optimum.
\begin{enumerate}
\item $(\dot{F}_{i+1} \cdot I_{n_{i+1}}) \cdot (\dot{F}_i \cdot I_{n_i}):$
Preaccumulation of both $F'_i$ and $F'_{i+1}$ in tangent mode and
		evaluation of the dense matrix product
$F'_{i+1} \cdot F'_i$ yields a computational cost of
$$
	\fma_{i+1,i}= n_{i} \cdot |E_{i}| +n_{i+1} \cdot |E_{i+1}|+m_{i+1} \cdot n_{i+1} \cdot n_i \; .
$$
It follows that $n_i \leq m_i$ and $n_{i+1} \leq m_{i+1}$ as the 
computational cost would otherwise be reduced by preaccumulation 
of either $F'_i$ or $F'_{i+1}$ (or both) in adjoint mode.
Superiority of homogeneous 
tangent mode follows immediately from
$n_i \leq m_i=n_{i+1} \leq m_{i+1}$ implying
		\begin{align*}
			n_{i} \cdot |E_{i}|+ n_i \cdot |E_{i+1}| &\leq 
			n_{i} \cdot |E_{i}|+ n_{i+1} \cdot |E_{i+1}| \\&<
	n_{i} \cdot |E_{i}| +n_{i+1} \cdot |E_{i+1}|+m_{i+1} \cdot n_{i+1} \cdot n_i \; .
		\end{align*}
\item $(I_{m_{i+1}} \cdot \bar{F}_{i+1}) \cdot (I_{m_i} \cdot \bar{F}_i):$
Preaccumulation of both $F'_i$ and $F'_{i+1}$ in adjoint mode and
evaluation of the dense matrix product
$F'_{i+1} \cdot F'_i$ yields a computational cost of
$$
	\fma_{i+1,i}=m_{i} \cdot |E_{i}|+m_{i+1} \cdot |E_{i+1}|+m_{i+1} \cdot n_{i+1} \cdot n_i \; .
$$
It follows that $n_i \geq m_i$ and $n_{i+1} \geq m_{i+1}$ as the 
computational cost would otherwise be reduced by preaccumulation 
of either $F'_i$ or $F'_{i+1}$ (or both) in tangent mode.
Superiority of homogeneous adjoint mode follows immediately from
$n_i \geq m_i=n_{i+1} \geq m_{i+1}$ implying
		\begin{align*}
			m_{i+1} \cdot |E_{i+1}|+ m_{i+1} \cdot |E_i| &\leq 
m_{i} \cdot |E_{i}|+m_{i+1} \cdot |E_{i+1}| \\
			&< m_{i} \cdot |E_{i}|+m_{i+1} \cdot |E_{i+1}|+m_{i+1} \cdot n_{i+1} \cdot n_i \; .
		\end{align*}
\item $(I_{m_{i+1}} \cdot \bar{F}_{i+1}) \cdot (\dot{F}_i \cdot I_{n_i}):$
Preaccumulation of $F'_i$ in tangent mode followed by 
		preaccumulation of $F'_{i+1}$ in adjoint mode and
		evaluation of the dense matrix product
$F'_{i+1} \cdot F'_i$ yields a computational cost of
$$
	\fma_{i+1,i}=n_{i} \cdot |E_{i}|+m_{i+1} \cdot |E_{i+1}|+m_{i+1} \cdot n_{i+1} \cdot n_i \; .
$$
It follows that $n_i \leq m_i$ and $n_{i+1} \geq m_{i+1}$ as the 
computational cost would otherwise be reduced by preaccumulation 
of either $F'_i$ in adjoint mode or by preaccumulation of $F'_{i+1}$ in 
tangent mode (or both).
This scenario turns out to be inferior to either homogeneous tangent or 
adjoint modes. For $n_i \leq m_{i+1}$ 
\begin{align*}
		n_{i} \cdot |E_{i}|+ n_i \cdot |E_{i+1}| &\leq 
		n_{i} \cdot |E_{i}|+ m_{i+1} \cdot |E_{i+1}| \\ &<
	n_{i} \cdot |E_{i}| +m_{i+1} \cdot |E_{i+1}|+m_{i+1} \cdot n_{i+1} \cdot n_i 
\end{align*}
while for 
$n_i \geq m_{i+1}$ 
		\begin{align*}
			m_{i+1} \cdot |E_{i+1}|+ m_{i+1} \cdot |E_i| &\leq 
			m_{i+1} \cdot |E_{i+1}|+ n_i \cdot |E_i| \\ &<
	n_{i} \cdot |E_{i}| +m_{i+1} \cdot |E_{i+1}|+m_{i+1} \cdot n_{i+1} \cdot n_i \; .
\end{align*}
\end{enumerate}
\paragraph{$1 \leq l\Rightarrow l+1$}
{\sc Generalized Dense Jacobian Chain Product Bracketing}  
inherits the {\em overlapping subproblems} property from 
{\sc Dense Jacobian Chain Product Bracketing}.
It adds two choices 
at each split location $i \leq k < j.$ 
Splitting at position $k$ implies the evaluation of 
$F'_{j,i}$ as $F'_{j,k+1} \cdot F'_{k,i}.$ 
In addition to both 
$F'_{j,k+1}$ and $F'_{k,i}$ being available 
there are the following two options: $F'_{k,i}$ is available and 
it enters the tangent $\dot{F}_{j,k+1} \cdot F'_{k,i}$ as argument;
$F'_{j,k+1}$ is available and it enters the adjoint $F'_{j,k+1} \cdot \bar{F}_{k,i}$ as argument.
All three options yield $F'_{j,i}$ and they correspond to 
Equation~\ref{eqn:dp2}~(b)--(d). 

The {\em optimal substructure} property remains to be shown. It implies
feasibility of tabulating solutions to all subproblems 
for constant-time lookup during the exhaustive search of the $3 \cdot l$ 
possible scenarios corresponding to the $l$ split locations.

Let the {\em optimal substructure} property not hold for an optimal
$\fma_{l+1,1}$ obtained at split location $1 \leq k < l+1.$ Three cases need 
to be distinguished that correspond to Equation~\ref{eqn:dp2}~(b)--(d). 
\begin{itemize}
\item[(b)] $\fma_{j,k+1}+\fma_{k,i} + m_j \cdot m_k \cdot n_i:$ 
The optimal substructure property holds for the preaccumulation of both 
$F'_{j,k+1}$ and $F'_{k,i}$ given as chains of length $\leq l.$ 
The cost of the dense matrix product $F'_{j,k+1} \cdot F'_{k,i}$ is independent
of the respective preaccumulation methods. 
For the {\em optimal substructure} property to not hold either the 
preaccumulation $F'_{j,k+1}$ or the preaccumulation of $F'_{k,i}$ must be 
suboptimal. However, replacement of this suboptimal 
preaccumulation method with the tabulated optimum would reduce the overall 
cost and hence yield the desired contradiction.
\item[(c)] $\fma_{j,k+1} + m_j \cdot \sum_{\nu=i}^k |E_\nu|:$
The optimal substructure property holds for the preaccumulation of $F'_{j,k+1}.$
The cost of the adjoint $F'_{j,k+1} \cdot \bar{F}_{k,i}$ is independent
of the preaccumulation method. 
The replacement of a suboptimal
preaccumulation of $F'_{j,k+1}$ with the tabulated optimum would reduce the 
overall cost and hence yield the desired contradiction.
\item[(d)] $\fma_{k,i} + n_i \cdot \sum_{\nu=k+1}^j |E_\nu|:$
The optimal substructure property holds for the preaccumulation of $F'_{k,i}.$
The cost of the tangent $\dot{F}_{j,k+1} \cdot F'_{k,i}$ is independent
of the preaccumulation method. 
The replacement of a suboptimal
preaccumulation of $F'_{k,i}$ with the tabulated optimum would reduce the 
overall cost and hence yield the desired contradiction.
\end{itemize}
\end{proof}
\paragraph{Example} We present examples for the previously discussed 
generalized dense Jacobian chain product of length two. Five configurations 
are considered with their solutions corresponding to the five instances of
the search space investigated in the proof of Theorem~\ref{the:GDJCPB}. Optimal
values are highlighted.
\begin{enumerate}
\item $n_1=2,$ $m_1=n_2=4,$ $m_2=8$, $|E_1|=|E_2|=100:$ 
\begin{align*}
&\fma\left (\dot{F}_2 \cdot (\dot{F}_1  \cdot I_{n_1})\right )={\bf 400}, \quad
\fma\left (\dot{F}_2 \cdot (I_{m_1} \cdot \bar{F}_1) \right )=600, \\
&\fma\left ((I_{m_2} \cdot \bar{F}_2) \cdot \bar{F}_1 \right )=1600, \quad
\fma\left ((\dot{F}_2 \cdot I_{n_2}) \cdot \bar{F}_1 \right )=1200, \\
&\fma\left ((\dot{F}_2 \cdot I_{n_2}) \cdot (I_{m_1} \cdot \bar{F}_1) \right )=864, \quad
\fma\left ((I_{m_2} \cdot \bar{F}_2) \cdot (I_{m_1} \cdot \bar{F}_1) \right )=1264, \\
&\fma\left ((I_{m_2} \cdot \bar{F}_2) \cdot (\dot{F}_1 \cdot I_{n_1}) \right )=1064, \quad
\fma\left ((\dot{F}_2 \cdot I_{n_2}) \cdot (\dot{F}_1 \cdot I_{n_1}) \right )= 664.
\end{align*}
\item $n_1=4,$ $m_1=n_2=2,$ $m_2=32$, $|E_1|=|E_2|=100:$ 
\begin{align*}
&\fma\left (\dot{F}_2 \cdot (\dot{F}_1  \cdot I_{n_1})\right )=800, \quad
\fma\left (\dot{F}_2 \cdot (I_{m_1} \cdot \bar{F}_1) \right )={\bf 600}, \\
&\fma\left ((I_{m_2} \cdot \bar{F}_2) \cdot \bar{F}_1 \right )=6400, \quad
\fma\left ((\dot{F}_2 \cdot I_{n_2}) \cdot \bar{F}_1 \right )=3400, \\
&\fma\left ((\dot{F}_2 \cdot I_{n_2}) \cdot (I_{m_1} \cdot \bar{F}_1) \right )=656, \quad
\fma\left ((I_{m_2} \cdot \bar{F}_2) \cdot (I_{m_1} \cdot \bar{F}_1) \right )=3656, \\
&\fma\left ((I_{m_2} \cdot \bar{F}_2) \cdot (\dot{F}_1 \cdot I_{n_1}) \right )=3856, \quad
\fma\left ((\dot{F}_2 \cdot I_{n_2}) \cdot (\dot{F}_1 \cdot I_{n_1}) \right )= 856.
\end{align*}
\item $n_1=8,$ $m_1=n_2=4,$ $m_2=2$, $|E_1|=|E_2|=100:$ 
\begin{align*}
&\fma\left (\dot{F}_2 \cdot (\dot{F}_1  \cdot I_{n_1})\right )=1600, \quad
\fma\left (\dot{F}_2 \cdot (I_{m_1} \cdot \bar{F}_1) \right )=1200, \\
&\fma\left ((I_{m_2} \cdot \bar{F}_2) \cdot \bar{F}_1 \right )={\bf 400}, \quad
\fma\left ((\dot{F}_2 \cdot I_{n_2}) \cdot \bar{F}_1 \right )=600, \\
&\fma\left ((\dot{F}_2 \cdot I_{n_2}) \cdot (I_{m_1} \cdot \bar{F}_1) \right )=864, \quad
\fma\left ((I_{m_2} \cdot \bar{F}_2) \cdot (I_{m_1} \cdot \bar{F}_1) \right )=664, \\
&\fma\left ((I_{m_2} \cdot \bar{F}_2) \cdot (\dot{F}_1 \cdot I_{n_1}) \right )=1064, \quad
\fma\left ((\dot{F}_2 \cdot I_{n_2}) \cdot (\dot{F}_1 \cdot I_{n_1}) \right )= 1264.
\end{align*}
\item $n_1=32,$ $m_1=n_2=2,$ $m_2=4$, $|E_1|=|E_2|=100:$ 
\begin{align*}
&\fma\left (\dot{F}_2 \cdot (\dot{F}_1  \cdot I_{n_1})\right )=6400, \quad
\fma\left (\dot{F}_2 \cdot (I_{m_1} \cdot \bar{F}_1) \right )=3400, \\
&\fma\left ((I_{m_2} \cdot \bar{F}_2) \cdot \bar{F}_1 \right )=800, \quad
\fma\left ((\dot{F}_2 \cdot I_{n_2}) \cdot \bar{F}_1 \right )={\bf 600}, \\
&\fma\left ((\dot{F}_2 \cdot I_{n_2}) \cdot (I_{m_1} \cdot \bar{F}_1) \right )=656, \quad
\fma\left ((I_{m_2} \cdot \bar{F}_2) \cdot (I_{m_1} \cdot \bar{F}_1) \right )=856, \\
&\fma\left ((I_{m_2} \cdot \bar{F}_2) \cdot (\dot{F}_1 \cdot I_{n_1}) \right )=3856, \quad
\fma\left ((\dot{F}_2 \cdot I_{n_2}) \cdot (\dot{F}_1 \cdot I_{n_1}) \right )= 3656.
\end{align*}
\item $n_1=4,$ $m_1=n_2=2,$ $m_2=4$, $|E_1|=|E_2|=100:$ 
\begin{align*}
&\fma\left (\dot{F}_2 \cdot (\dot{F}_1  \cdot I_{n_1})\right )=800, \quad
\fma\left (\dot{F}_2 \cdot (I_{m_1} \cdot \bar{F}_1) \right )=600, \\
&\fma\left ((I_{m_2} \cdot \bar{F}_2) \cdot \bar{F}_1 \right )=800, \quad
\fma\left ((\dot{F}_2 \cdot I_{n_2}) \cdot \bar{F}_1 \right )=600, \\
&\fma\left ((\dot{F}_2 \cdot I_{n_2}) \cdot (I_{m_1} \cdot \bar{F}_1) \right )={\bf 432}, \quad
\fma\left ((I_{m_2} \cdot \bar{F}_2) \cdot (I_{m_1} \cdot \bar{F}_1) \right )=632, \\
&\fma\left ((I_{m_2} \cdot \bar{F}_2) \cdot (\dot{F}_1 \cdot I_{n_1}) \right )=832, \quad
\fma\left ((\dot{F}_2 \cdot I_{n_2}) \cdot (\dot{F}_1 \cdot I_{n_1}) \right )= 632.
\end{align*}
\end{enumerate}

\section{Implementation and Numerical Results} \label{sec:impl}

Our reference implementation can be downloaded from 
\begin{center}
	\tt
	www.github.com/un110076/ADMission/GDJCPB
\end{center}
together with the sample problems referred to in this section.
Two executables are provided:
\verb!gdjcpb_generate.exe! generates problem 
instances randomly for a given length \verb!len! of the chain and upper
bound \verb!max_m_n! on the number of rows and columns of the individual 
factors. The output can be redirected into a text file which serves as input 
to \verb!gdjcpb_solve.exe!.  
The latter computes one solution to the given 
problem instance. This solution is compared with the costs of the 
homogeneous tangent, adjoint and preaccumulation methods. The latter
implies a solution of the resulting 
{\sc Dense Jacobian Chain Product Bracketing} problem.

The source code is written in simple C++. It should compile under arbitrary 
operating systems assuming availability of a C++14 standard compliant compiler. 
The \verb!Makefile! 
provided covers Linux and \verb!g++! (e.g, version 7.4.0). 
A \verb!README! contains essential 
instructions for building and running.

\begin{table}
\centering
\small
\begin{tabular}{|c|c|c|c|c|c|}
\hline
\verb!len! & \verb!max_mn! & Tangent & Adjoint & Preaccumulation & Optimum \\
\hline
	10 & 10 & 3,708 & 5,562 & \bf 2,618 & \bf 1,344 \\
50 & 50 & \bf 1,283,868 & 1,355,194 & 1,687,575 & \bf 71,668 \\
100 & 100 & \bf 3,677,565 & 44,866,293 & 40,880,996 & \bf 1,471,636\\
250 & 250 & \bf 585,023,794 & 1,496,126,424 & 1,196,618,622& \bf 9,600,070 \\
500 & 500& 21,306,718,862 & 19,518,742,454 & \bf 1,027,696,225 & \bf 149,147,898\\
\hline
\end{tabular}
	\caption{Test Results: Cost in $\fma;$ Superiority of the solutions to {\sc Generalized Dense Jacobian Chain Bracketing} are quantified as the ratios of the highlighted entries in each row.} \label{tab:res}
\end{table}
\paragraph{Example} Running 
\verb!gdjcpb_generate.exe 3 3! yields, for example, 
\begin{lstlisting}
3
3 3 29
1 3 14
2 1 7
\end{lstlisting}
describing the problem instance 
$
F'=F'_3 \cdot F'_2 \cdot F'_1,
$
where \\
\\
$F'_1 \in \R^{3 \times 3} \rightarrow G_1=(V_1,E_1):~|E_1|=29 $\\
$F'_2 \in \R^{1 \times 3} \rightarrow G_2=(V_2,E_2):~|E_2|=14 $\\
$F'_3 \in \R^{2 \times 1} \rightarrow G_3=(V_3,E_3):~|E_3|=7 \; .$ \\
\\
Let this problem description be stored in the text file \verb!problem.txt!.
Running \\
\begin{verbatim}
gdjcpb_solve.exe problem.txt 
\end{verbatim} $\;$ \\
generates the following output:
\begin{lstlisting}
Dynamic Programming Table:
fma_{1,1}=87; Split=0; Operation=Tangent
fma_{2,2}=14; Split=0; Operation=Adjoint
fma_{2,1}=43; Split=1; Operation=Adjoint
fma_{3,3}=7; Split=0; Operation=Tangent
fma_{3,2}=27; Split=2; Operation=Preaccumulation
fma_{3,1}=56; Split=2; Operation=Preaccumulation

Optimal Cost=56

Cost of homogeneous tangent mode=150
Cost of homogeneous adjoint mode=100
Cost of optimal homogeneous preaccumulation=108+15=123
\end{lstlisting}
$F'_1$ is optimally accumulated in tangent mode at the expense of $3\cdot 29=87 \fma$ (similarly, 
$F'_2$ in adjoint mode at $1\cdot 14=14 \fma$ and
$F'_3$ in tangent mode at $1\cdot 7=7 \fma.$ Splitting is not applicable 
(\lstinline{Split=0}).
The optimal method to compute 
$F'_{2,1}$ uses adjoint mode as $F'_2 \cdot \bar{F}_1$ at cost
$14+1 \cdot 29=43 \fma.$
Preaccumulation of $F'_2$ and $F'_3$ followed by the dense matrix product
$F'_3 \cdot F'_2$ turns out to be the optimal method for computing $F'_{3,2}.$
The entire problem instance is evaluated optimally as
$$
F'=(\dot{F}_3 \cdot I_1) \cdot ((I_1 \cdot \bar{F}_2) \cdot \bar{F}_1)
$$
yielding a computational cost of $7\cdot 1 + (14+29)\cdot 1 + 2 \cdot 1 \cdot 3=56 \fma.$ 

Homogeneous tangent mode
$$
F':=\dot{F}_3 \cdot (\dot{F}_2 \cdot (\dot{F}_1 \cdot I_{n_1}))
$$
yields a cost of $n_1 \cdot \sum_{i=1}^3 |E_i|=3 \cdot (29+14+7)=150 \fma.$
Homogeneous adjoint mode
$$
F':=((I_{m_3} \cdot \bar{F}_3) \cdot \bar{F}_2) \cdot \bar{F}_1
$$
yields a cost of $m_3 \cdot \sum_{i=1}^3 |E_i|=2 \cdot (29+14+7)=100 \fma.$
Optimal preaccumulation of $F'_i$ for $i=1,2,3$ takes
$\sum_{i=1}^3 |E_i| \cdot \min(m_i,n_i)=1 \cdot 7 + 1 \cdot 14 + 3 \cdot 29=108 \fma$ followed by optimal bracketing as 
$$
F'=F'_3 \cdot (F'_2 \cdot F'_1)
$$
adding $9+6=15 \fma$ and
yielding a total cost of the optimal homogeneous preaccumulation method of $108+15=123 \fma.$
The dynamic programming solution of the 
{\sc Generalized Dense Jacobian Chain Product Bracketing} problem
yields an improvement of nearly $50$ percent over homogeneous adjoint mode.

In Table~\ref{tab:res} we present further results for problem instances of
growing size generated by calling
\verb!gdjcpb_generate.exe len max_mn.!
The dynamic programming solutions improve the best homogeneous
method by factors between two and sixty. Full specifications of all five test 
problems can be found in the github repository.

\section{Conclusion and Outlook}

This paper generalizes prior work on (Jacobian) matrix chain products in the
context of algorithmic differentiation \cite{Griewank2008EDP,Naumann2012TAo}. Tangents and adjoints of 
differentiable subprograms of numerical simulation programs are typically available rather than the corresponding local Jacobian matrices. Optimal combination of tangents and adjoints yield
sometimes impressive reductions of the overall operations count (factors of 
up to sixty are reported in Section~\ref{sec:impl}). Dynamic programming makes
the underlying abstract combinatorial problem formulation computationally 
tractable. 

Applicability of the algorithmic results of this paper to real world
applications requires further generalization. Rigorous minimization of
the computational cost of an algorithmic differentiation task
must be based on 
information about elemental data dependences and resulting Jacobian sparsity 
patterns. Constraints on the available 
persistent memory need to be taken into account. Coarser grain data dependence
patterns yield matrix DAGs rather than matrix chains. See below for further
illustration.

\paragraph{Exploitation of Local DAG Structure and Jacobian Sparsity} 

Exploitation of sparsity of the $F_i$ in Equation~(\ref{eqn:jcp}) impacts the 
computational cost estimate for their preaccumulation. 
For example, the nonzero entries of a diagonal matrix 
$F'_i \in \R^{n_i \times n_i}$ can be obtained at the expense of $|E_i|$ \fma\
in either tangent or adjoint modes.
Various Jacobian compression 
techniques based on coloring of different representations of the sparsity 
patterns as graphs have been proposed for general Jacobian sparsity patterns 
\cite{Gebremedhin2005WCI}. The minimization of 
the overall computational cost becomes intractable as a consequence of 
intractability of the underlying coloring problems.

Further exploitation of data dependence patterns through structural properties
of the concatenation of the local DAGs may lead to further decrease of the 
computational cost. Vertex, edge, and face elimination techniques have been 
proposed to allow for applications of the chain rule beyond Jacobian chain multiplication \cite{Naumann2004Oao}. For example, the following sparse Jacobian chain product was used in
\cite{Naumann2020CSC} to illustrate superiority of vertex elimination \cite{Griewank1991OtC}:
$$
        \begin{pmatrix}
        m^2_{0,0} & 0 & 0 \\
         0 & m^2_{1,1} & m^2_{1,2} \\
\end{pmatrix}
\begin{pmatrix}
        m^1_{0,0} & m^1_{0,1} & 0 \\
         0 & m^1_{1,1} & m^1_{1,2} \\
         0 & 0  & m^1_{2,2} \\
\end{pmatrix}
\begin{pmatrix}
        m^0_{0,0} &0  \\
         m^0_{1,0} & 0 \\
         0 & m^0_{2,1} \\
\end{pmatrix} \; .
$$
It is straight forward to verify that both bracketings yield a computational 
cost of $9 \fma.$ Full exploitation of distributivity enables computation
of the resulting matrix as
$$
\begin{pmatrix}
        m^2_{0,0} (m^1_{0,0} m^0_{0,0} +m^1_{0,1} m^0_{1,0}) & 0 \\
        m^2_{1,1} m^1_{1,1} m^0_{1,0} & (m^2_{1,1} m^1_{1,2} + m^2_{1,2} m^1_{2,2}) m^0_{2,1} 
\end{pmatrix} 
$$
at the expense of only $8 \fma.$ 

\paragraph{Adding Memory Constraints} 
Let $F=F_3 \circ F_2 \circ F_1 : \R^8 \rightarrow \R$ such that
$F_1 : \R^8 \rightarrow \R^4,$ $F_2 : \R^4 \rightarrow \R^2,$ 
$F_3 : \R^2 \rightarrow \R^1$ and $|E_i|=16$ for $i=1,2,3.$
Execution of \verb!gdjcpb_solve.exe! for a corresponding problem specification
yields the following output:
\begin{lstlisting}
Dynamic Programming Table:
fma_{1,1}=64; Split=0; Operation=Adjoint
fma_{2,2}=32; Split=0; Operation=Adjoint
fma_{2,1}=64; Split=1; Operation=Adjoint
fma_{3,3}=16; Split=0; Operation=Adjoint
fma_{3,2}=32; Split=2; Operation=Adjoint
fma_{3,1}=48; Split=1; Operation=Adjoint

Optimal Cost=48

Cost of homogeneous tangent mode=384
Cost of homogeneous adjoint mode=48
Cost of optimal homogeneous preaccumulation=112+40=152
\end{lstlisting}
Obviously, homogeneous adjoint mode turns out to be optimal, which is also recovered
by the dynamic programming algorithm. 
Let the persistent memory requirement of adjoint mode applied to
$F_i$ be estimated as $|E_i|.$ The total memory requirement of homogeneous 
adjoint mode is equal to $3 \cdot 16=48.$ Let the size of the available 
persistent memory requirement be bounded from above by $\hat{M}=40.$ Homogeneous 
adjoint mode becomes infeasible. 

Feasible alternatives include the preaccumulation of $F'_3$ in tangent mode
with no extra persistent memory required and followed by adjoint mode applied
to $F_2$ and $F_1$ yielding
$$
F'=((\dot{F}_3 \cdot I_2) \cdot \bar{F}_2) \cdot \bar{F}_1
$$
at the expense of $2 \cdot 16 + 1 \cdot 16 + 1 \cdot 16=64 \fma$ and with 
feasible persistent memory requirement of $16+16=32.$

The {\em optimal substructure} property does not hold anymore. 

\paragraph{From Matrix Chains to Matrix DAGs}

Let 
$$
F=\begin{pmatrix}
	F_2 \circ F_1 \\
	F_3 \circ F_1
\end{pmatrix} : \R^2 \rightarrow \R^5
$$
such that $F_1 : \R^2 \rightarrow \R^4,$
$F_2 : \R^4 \rightarrow \R^1,$
$F_3 : \R^4 \rightarrow \R^4$
and $|E_i|=16$ for $i=1,2,3.$ Assuming availability of sufficient persistent 
memory homogeneous adjoint mode turns out to be optimal
for $F_2 \circ F_1.$ The Jacobian of 
$F_3 \circ F_1$ is optimally computed in homogeneous tangent mode which
yields a conflict for $F'_1.$ Separate optimization of the two Jacobian chain 
products $F'_2 \cdot F'_1$ and $F'_3 \cdot F'_1$ yields a cumulative 
computational cost of $1 \cdot (16+16) + 2 \cdot (16+16)=96 \fma.$ 
A better solution is
$$
F'=\begin{pmatrix}
\dot{F}_2 \cdot (\dot{F}_1 \cdot F_2) \\
(I_1 \cdot \bar{F}_3) \cdot (\dot{F}_1 \cdot F_2) 
\end{pmatrix}
$$
yielding a slight decrease in the computational cost to
$2 \cdot 16 + 2 \cdot 16 + 1 \cdot 16 + 1 \cdot 4 \cdot 8=88 \fma.$ 
More significant savings can be expected for less simple DAGs.


\begin{thebibliography}{10}

\bibitem{Adams1979THG}
D.~Adams.
\newblock {\em The Hitchhiker's Guide to the Galaxy}.
\newblock Pan Books, 1979.

\bibitem{Baur1983TCo}
W.~Baur and V.r Strassen.
\newblock The complexity of partial derivatives.
\newblock {\em Theoretical Computer Science}, 22:317--330, 1983.

\bibitem{Bellman1957DP}
R.~Bellman.
\newblock {\em {Dynamic Programming}}.
\newblock Dover Publications, 1957.

\bibitem{Gauss1801DA}
tr. A.~Clarke C.~Gauss.
\newblock {\em Disquisitiones Arithmeticae}.
\newblock Yale University Press, 1965.

\bibitem{Chen2012AIP}
J.~Chen, P.~Hovland, T.~Munson, and J.~Utke.
\newblock An integer programming approach to optimal derivative accumulation.
\newblock In S.~Forth, P.~Hovland, E.~Phipps, J.~Utke, and A.~Walther, editors,
  {\em Recent Advances in Algorithmic Differentiation}, volume~87 of {\em
  Lecture Notes in Computational Science and Engineering}, pages 221--231.
  Springer, Berlin, 2012.

\bibitem{Forth2004JCG}
S.~Forth, M.~Tadjouddine, J.~Pryce, and J.~Reid.
\newblock Jacobian code generated by source transformation and vertex
  elimination can be as efficient as hand-coding.
\newblock {\em {ACM} Transactions on Mathematical Software}, 30(3):266--299,
  2004.

\bibitem{Garey1979CaI}
M.~Garey and D.~Johnson.
\newblock {\em Computers and Intractability: A Guide to the Theory of
  NP-Completeness (Series of Books in the Mathematical Sciences)}.
\newblock W. H. Freeman, first edition edition, 1979.

\bibitem{Gebremedhin2005WCI}
A.~Gebremedhin, F.~Manne, and A.~Pothen.
\newblock What color is your {J}acobian? {G}raph coloring for computing
  derivatives.
\newblock {\em SIAM Review}, 47(4):629--705, 2005.

\bibitem{Godbole1973}
S.~{Godbole}.
\newblock On efficient computation of matrix chain products.
\newblock {\em IEEE Transactions on Computers}, C-22(9):864--866, Sep. 1973.

\bibitem{Griewank1992ALG}
A.~Griewank.
\newblock Achieving logarithmic growth of temporal and spatial complexity in
  reverse automatic differentiation.
\newblock {\em Optimization Methods and Software}, 1:35--54, 1992.

\bibitem{Griewank2002DJS}
A.~Griewank and C.~Mitev.
\newblock Detecting {J}acobian sparsity patterns by {B}ayesian probing.
\newblock {\em Mathematical Programming, Ser.~A}, 93(1):1--25, 2002.

\bibitem{Griewank2003AJa}
A.~Griewank and U.~Naumann.
\newblock Accumulating {J}acobians as chained sparse matrix products.
\newblock {\em Mathematical Programming, Ser.~A}, 95(3):555--571, 2003.

\bibitem{Griewank1991OtC}
A.~Griewank and S.~Reese.
\newblock On the calculation of {J}acobian matrices by the {M}arkowitz rule.
\newblock In A.~Griewank and G.~Corliss, editors, {\em Automatic
  Differentiation of Algorithms: Theory, Implementation, and Application},
  pages 126--135. SIAM, Philadelphia, PA, 1991.

\bibitem{Griewank2008EDP}
A.~Griewank and A.~Walther.
\newblock {\em Evaluating Derivatives: {P}rinciples and Techniques of
  Algorithmic Differentiation}.
\newblock Number 105 in Other Titles in Applied Mathematics. SIAM,
  Philadelphia, PA, 2nd edition, 2008.

\bibitem{Hossain2002SIi}
S.~Hossain and T.~Steihaug.
\newblock Sparsity issues in the computation of {J}acobian matrices.
\newblock In T.~Mora, editor, {\em Proceedings of the International Symposium
  on Symbolic and Algebraic Computing (ISSAC)}, pages 123--130, New York, NY,
  2002. ACM.

\bibitem{Naumann2020CSC}
U.~Naumann.
\newblock {\em On Sparse Matrix Chain Products}.
\newblock 2020 Proceedings of the SIAM Workshop on Combinatorial Scientific Computing, 118--127.

\bibitem{Naumann2004Oao}
U.~Naumann.
\newblock Optimal accumulation of {J}acobian matrices by elimination methods on
  the dual computational graph.
\newblock {\em Mathematical Programming, Ser.~A}, 99(3):399--421, 2004.

\bibitem{Naumann2008OJa}
U.~Naumann.
\newblock Optimal {J}acobian accumulation is {NP}-complete.
\newblock {\em Mathematical Programming, Ser.~A}, 112(2):427--441, 2008.

\bibitem{Naumann2012TAo}
U.~Naumann.
\newblock {\em The Art of Differentiating Computer Programs: {A}n Introduction
  to Algorithmic Differentiation}.
\newblock Number~24 in Software, Environments, and Tools. SIAM, Philadelphia,
  PA, 2012.

\bibitem{Naumann2008Ove}
U.~Naumann and Y.~Hu.
\newblock Optimal vertex elimination in single-expression-use graphs.
\newblock {\em ACM Transactions on Mathematical Software}, 35(1):2:1--2:20,
  July 2008.

\bibitem{Naumann2018LMA}
U.~Naumann and K.~Leppkes.
\newblock Low-Memory Algorithmic Adjoint Propagation.
\newblock 2018 Proceedings of the SIAM Workshop on Combinatorial Scientific Computing, 1--10.

\bibitem{Pryce2008FAD}
J.~Pryce and M.~Tadjouddine.
\newblock Fast automatic differentiation {J}acobians by compact {LU}
  factorization.
\newblock {\em SIAM Journal on Scientific Computing}, 30(4):1659--1677, 2008.

\end{thebibliography}
\end{document}